\documentclass[12pt,a4paper]{article}
\usepackage{amssymb}
\usepackage{amsmath, enumerate,amsthm,fullpage,wasysym, bbm,color}
\newtheorem{theorem}{Theorem}

\newtheorem{conjecture}[theorem]{Conjecture}
\newtheorem{corollary}[theorem]{Corollary}

\newtheorem{definition}[theorem]{Definition}
\newtheorem{example}[theorem]{Example}

\newtheorem{lemma}[theorem]{Lemma}

\newtheorem{proposition}[theorem]{Proposition}
\newtheorem{remark}[theorem]{Remark}

\newcommand{\real}{\mathbb{R}}

\newcommand{\calC}{\mathcal{C}}

\begin{document}
\title{
On the Law of Free Subordinators
}
\author{Octavio Arizmendi\footnote{Supported by funds of R. Speicher from the Alfried Krupp von Bohlen und Halbach Stiftung. E-mail: arizmendi@math.uni-sb.de} \\ Universit\"{a}t des Saarlandes, FR $6.1-$Mathematik,\\ 66123 Saarbr\"{u}cken, Germany \\ \\ Takahiro Hasebe\footnote{Supported by Global COE Program at Kyoto University. E-mail: thasebe@math.kyoto-u.ac.jp}\\Graduate School of Science,  Kyoto University,\\  Kyoto 606-8502, Japan \\ \\ Noriyoshi Sakuma\footnote{Supported by  JSPS KAKENHI Grant Number 24740057. E-mail: sakuma@auecc.aichi-edu.ac.jp} \\ Department of Mathematics, Aichi University of Education\\ 1 Hirosawa, Igaya-cho, Kariya-shi, 448-8542, Japan\\} 

\date{\today}
\maketitle
\begin{abstract}
We study the freely infinitely divisible distributions that appear as the laws of free subordinators.
This is the free analog of classically infinitely divisible distributions supported on $[0,\infty)$, called the free regular measures. 

 We prove that the class of free regular measures is closed under the free multiplicative convolution, $t$th boolean power for $0\leq t\leq 1$, $t$th free multiplicative power for $t\geq 1$ and weak convergence. 

In addition, we show that a symmetric distribution is freely infinitely divisible if and only if its square can be represented as the free multiplicative convolution of a free Poisson and a free regular measure. 

This gives two new explicit examples of distributions which are infinitely divisible with respect to both classical and free convolutions: $ \chi^2(1)$ and $ F(1,1)$.
 Another consequence is that the free commutator operation preserves free infinite divisibility. 

\end{abstract}
\newpage

\section{Introduction}

A one dimensional subordinator $(X_{t})_{t\geq 0}$ is a L\'evy process whose increments are always nonnegative.
The marginal distributions $(\mu_{t})_{t\geq0}$ of a subordinator $(X_{t})_{t\geq 0}$
are infinitely divisible and
their L\'evy-Khintchine representations have {\it regular} forms for any $t\geq0$:
\begin{equation}\label{eq01}
\mathcal{C}^{\ast}_{\mu_{t}}(z):=\log \left(\int_{\real}e^{izx} \mu_{t}(dx)\right)
=i t \eta' z + t \int_{(0,\infty)}(e^{izx}-1)\nu(dx),
\end{equation}
where the drift term $\eta'$ satisfies $\eta'\geq0$ and the L\'evy measure $\nu$ satisfies 
$\int_{(0,\infty)}(1\wedge x)\nu(dx)<\infty$ and $\nu((-\infty,0])=0$.
Poisson processes, positive stable processes and Gamma processes are typical examples.
Subordinators have been broadly studied, see for example, Bertoin\cite{Be99} and Sato\cite{Sato}. 
For applications in financial modeling, see Cont and Tankov\cite{CoTa}. 
A matrix valued extension has been considered in Barndorff-Nielsen and P\'{e}rez-Abreu\cite{PABN}.

A crucial property is that the class of infinitely divisible distributions with regular L\'evy-Khintchine representations is closed under $\ast$-convolution powers.
Namely, a $\ast$-infinitely divisible distribution $\mu$ has a regular L\'evy-Khintchine representation if and only if
$\mu_{t}=\mu^{\ast t}$ is concentrated on $[0,\infty)$ for all $t>0$.
See for details Theorem 24.11 in  p.146 of the book by Sato\cite{Sato}.

In free probability theory, the free convolution or $\boxplus$-convolution was introduced 
by Voiculescu \cite{Vo86} in order to describe the sum of free random variables. The main analytic tool for the study of free convolution
is the so-called Voiculescu's  R-transform or free cumulant transform,
denoted here by $\mathcal{C}^{\boxplus}_{\mu}(z)$.
The basic property of the free cumulant transform is that it linearizes the free convolution:
$$
\mathcal{C}^{\boxplus}_{\mu\boxplus\rho}(z) = \mathcal{C}^{\boxplus}_{\mu}(z)+\mathcal{C}^{\boxplus}_{\rho}(z).
$$

Similarly to the classical case, one can define  free L\'evy processes and free infinite divisibility with respect to free convolution. One  obtains the corresponding  L\'evy-Khintchine representation for the free cumulant transform.
This representation is also given in terms of a characteristic triplet $(\eta,a,\nu)$ 
that satisfies the same properties as in the classical case. 
This produces a bijection $\Lambda$, first introduced by Bercovici and Pata \cite{Be-Pa99},
between classically and freely infinitely divisible distributions. 

In this context, we can also define the free counterpart of laws of subordinators, that is $\rho_{t} = \Lambda(\mu_{t})$, where $\mu_t$ has the regular form (\ref{eq01}).
The free cumulant transforms of the laws $(\rho_{t})_{t\geq 0}= (\rho^{\boxplus t})_{t\geq 0}$ have the \emph{free regular representations} 
\begin{equation}
\mathcal{C}^{\boxplus}_{\rho_{t}}(z)=t \eta' z+t\int_{\mathbb{R}}\left( \frac{1}{1-zx}-1 \right) \nu \left(
dx\right) ,\quad z\in \mathbb{C}_{-}\text{,}  \label{levykintchine libre}
\end{equation}
where $(\eta',\nu)$ is the pair of (\ref{eq01}) with the same conditions: $\eta'\geq0$, 
$\int_{(0,\infty)}(1\wedge x)\nu(dx)<\infty$ and $\nu((-\infty,0])=0$. It is readily seen that this class is closed under the convolution $\boxplus$.

Let us note here an important difference 
between classically and freely infinitely divisible distributions on the cone $[0,\infty)$.
 Any classically infinitely divisible distribution supported on $[0,\infty)$ satisfies that $\mu_{t}=\mu^{\boxplus t}$ is concentrated on $[0,\infty)$ for all time $t>0$, and thus have a regular representation. However, we can easily find a freely infinitely divisible distribution concentrated on $[0,\infty)$, but $\mu_{t}=\mu^{\boxplus t}$ is not concentrated on 
 $[0,\infty)$ for all time $t>0$. 
For example, the semicircle distribution
with mean $2$ and variance $1$. If we construct a free L\'evy process from this distribution, 
the laws $\mu_{t}$ for $t\geq1$ concentrate on $[0,\infty)$ 
but do not for $0<t<1$, 
see \cite{Sa10} for more details.  Thus, in this sense the correct counterpart of the class of $\ast$-infinitely divisible distributions supported on $[0,\infty)$ is the class of free regular measures.


The main purpose of this paper is to show strong closure properties of the class of free regular measures under different convolutions as well as several important consequences.
More specifically, we prove that the class of free regular measures is closed 
not only under free additive convolution $\boxplus$
but also under free multiplicative convolution $\boxtimes$ and boolean convolution powers.

As a first important consequence, we characterize the laws of free subordinators in terms of free regularity. 
More precisely, $(Z_{t})_{t\geq0}$ is a free L\'evy process 
such that the distribution of $Z_{t}-Z_{s}$ has non-negative support if and only if the law $Z_{1}$ is free regular.

As a second important consequence,
if $X$ and $Y$ are two free independent random variables with 
free regular distributions, then $X^{1/2}YX^{1/2}$ also follows a free regular distribution, which is not true in the classical case.
See Example 11.3 in Chapter $2$ of the book by Steutel and van Harn \cite{StVH}.

Other results and the organization of this paper are as follows. 
First, we state the main theorems in Section 2.
In Section 3 we review some basic theory of non-commutative probability. We recall free additive and multiplicative convolutions and the analytic tools to calculate them. We state basic results on free infinite divisibility such as L\'evy-Khintchine representations and the Bercovici-Pata bijection $\Lambda$. Also, we explain boolean additive convolution and recall the boolean-to-free Bercovici-Pata bijection $\mathbb{B}$. Section 4 is devoted to the description of different characterizations of free regular measures. 
In Section 5 we derive, using the characterizations of Section 3, closure properties as explained 
in Theorem \ref{closur}.
In Section 6 we essentially prove Theorem \ref{main3} below, which in particular shows that the square of a symmetric freely infinitely divisible distribution is freely infinitely divisible.
We partially show that, for selfadjoint operators, the free infinite divisibility is preserved under the free commutator operation. This fact is fully proved in Appendix with combinatorial techniques. Finally, in Section 7 we gather examples using results of previous sections and present open problems regarding these examples.
At the end of paper, we give an appendix where combinatorial interpretation of Theorem \ref{main3} is discussed. It contributes to study free commutators.

\section{Main results}
Let $\mathcal{M}$ be the class of all Borel probability measures on the real line $\mathbb{R}$ and let $\mathcal{M}^{+}$ be the subclass of $\mathcal{M}$ consisting of probability measures with support on $\mathbb{R}_{+}=[0,\infty )$. Also, for two probability measures $\mu,\nu\in\mathcal{M}$, we denote by $\mu\ast\nu$, $\mu\boxplus\nu$ and $\mu\uplus\nu$ the classical, free and boolean additive convolutions, respectively. When $\nu\in\mathcal{M}^+$ we denote by $\mu\boxtimes\nu$ the free multiplicative convolution. They will be defined precisely in Section \ref{prel}. 

Let $I^{\ast}$ be the class of all classically infinitely divisible distributions and $I^{\boxplus}$ be the class of all freely infinitely divisible distributions. 
An important subclass of $I^{\ast }$ is the class of infinitely divisible measures supported on $\mathbb{R}_{+}$, that is, $I^{\ast }\cap\mathcal{M}^{+}$. This class has regular L\'evy-Khintchine representations.

Free regular measures are the free analogue of $I^{\ast }\cap\mathcal{M}^{+}$. 
More precisely, let $I_{r+}^{\boxplus }:=\Lambda(I^{\ast }\cap\mathcal{M}^{+})$, where $\Lambda:I^*\rightarrow I^\boxplus$ is the Bercovici-Pata bijection, which is defined in Section $3$. This class  $ I_{r+}^{\boxplus }$ was first considered in \cite{PeSa} in connection to free multiplicative mixtures of the Wigner distribution. It is remarkable that $ I_{r+}^{\boxplus } \subset I^{\boxplus }\cap\mathcal{M}^{+}$ but $ I_{r+}^{\boxplus }\neq I^{\boxplus }\cap\mathcal{M}^{+}$; the Bercovici-Pata bijection can send measures with support larger than $\mathbb{R}_+$ to measures concentrated on $[0,\infty)$. 

The main results are as follows. First, we will see that $I_{r+}^{\boxplus }$ describes the distributions of free L\'{e}vy processes with positive increments, that we will call \emph{free subordinators}. For free L\'evy processes, contrary to the classical, boolean and monotone cases, the positivity of the marginal distribution at time $t=1$ does not imply the positivity of all increments.

Second, $ I_{r+}^{\boxplus }$ behaves well with respect to various operations in non-commutative probability. More specifically, we are able to prove the following. 

\begin{theorem}\label{closur}
Let $\mu, \nu $ be free regular measures and let $\sigma $ be a freely
infinitely divisible distribution. Then the following properties hold. 

\begin{enumerate}[{\rm (1)}]
	\item
	$\mu \boxtimes \nu $ is free regular.
	\item
	$\mu ^{\boxtimes t}$ is free regular for $t \geq 1.$
	\item
	$\mu ^{\uplus t}$ is free regular for $0\leq t\leq 1.$
	\item
	$\mu \boxtimes \sigma $ is freely infinitely divisible.
\end{enumerate}
\end{theorem}
Of particular interest is the fact that $ I_{r+}^{\boxplus }$ is closed under free multiplicative convolution. It was proved by Belinschi and Nica \cite {BN08} that the boolean-to-free Bercovici-Pata bijection $\mathbb{B}$ is a homomorphism with respect to free multiplicative convolution. This suggested strongly that free infinite divisibility was preserved under free multiplicative convolution. Surprisingly, this is not true, even if we restrict to measures in $\mathcal{M}^+$. Therefore, $ I_{r+}^{\boxplus } $ is a natural class to consider, since it solves this apparent flaw. 

The final result shows that  if a symmetric random variable $X$ has a distribution in $I^\boxplus$, so does the square $X^2$.  
This result is quite surprising since there is no analog in the classical world.  We describe this result precisely below. For $p \geq 0$, let $\mu^p$ denote the probability measure on $[0,\infty)$ induced by the map $x \mapsto |x|^p$. 

\begin{theorem}\label{main3} Let $\mu$ be a symmetric measure and $m$ be the free Poisson law with density $\frac{1}{2\pi}\sqrt{\frac{4-x}{x}}$. 

\begin{enumerate}[{\rm (1)}]
\item
If $\mu$ is $\boxplus$-infinitely divisible, then there is a free regular measure $\sigma$ such that $\mu^2 = m \boxtimes \sigma$. In particular, $\mu^2 \in I_{r+}^\boxplus$. 
Conversely, if $\sigma$ is free regular, then $\text{\normalfont Sym}\left((m \boxtimes \sigma)^{1/2}\right)$ is $\boxplus$-infinitely divisible distribution, where $\text{\normalfont Sym}(\nu)$ is the symmetrization of $\nu \in \mathcal{M}^+$: $\text{\normalfont Sym}(\nu)(dx):=\frac{1}{2}\left(\nu(dx)+\nu(-dx)\right)$.  

\item
If  $\mu$ is a compound free  Poisson with rate $\lambda$ and jump distribution $\nu$, then $\sigma$ from (1) is also a compound free Poisson with rate $\lambda$ and jump distribution $\nu ^{2}$. 
\end{enumerate}

\end{theorem}

 As a consequence we find two new explicit examples of measures which are infinitely divisible in both free and classical senses
: $\chi^2(1)$ and $F(1,1).$ To the best of our knowledge, apart from these two examples, there are only three known measures with this property: the normal law, the Cauchy distribution and the free $1/2$ stable law.
 
 Secondly, we get as a byproduct that the free commutator of freely infinitely divisible measures is also infinitely divisible.

\section{Preliminaries}\label{prel}

\subsection{Analytic tools for free convolutions}

Following \cite{VoDyNi92}, we recall that a pair $(\mathcal{A},\varphi )$ is called a $W^{\ast }$-\textit{probability space} if $\mathcal{A}$ is a von Neumann algebra and $ \varphi $ is a normal faithful trace. A family of
unital von Neumann subalgebras $\left\{ \mathcal{A}_{i}\right\} _{i\in I}\mathcal{\subset A}$ is said to be \textit{free} if $\varphi (a_{1}\cdot \cdot
\cdot a_{n})=0$ whenever $\varphi (a_{j})=0,a_{j}\in \mathcal{A}_{i_{j}},$ and $i_{1}\neq
i_{2},i_{2}\neq i_{3},...,i_{n-1}\neq i_{n}.$ A self-adjoint operator $X$\textit{\ is said to be
affiliated with} $ \mathcal{A}$ if $f(X)\in \mathcal{A}$ for any bounded Borel function $f$ on
$ \mathbb{R}$. In this case it is also said that $X$ is a (non-commutative) \textit{random
variable}. Given a self-adjoint operator $X$ affiliated with $ \mathcal{A}$, the
\textit{distribution} of $X$ is the unique measure $\mu _{X}$ in $\mathcal{M}$ satisfying
\begin{equation*}
\varphi (f(X))=\int_{\mathbb{R}}f(x)\mu _{X}(\mathrm{d}x)
\end{equation*}
for every Borel bounded function $f$ on $\mathbb{R}$. \ If $\left\{ \mathcal{ A}_{i}\right\}
_{i\in I}$ is a family of free unital von Neumann subalgebras and $X_{i}$ is a random variable
affiliated with $\mathcal{A}_{i}$ for each $ i\in I$, then the \textit{random variables}
$\left\{ X_{i}\right\} _{i\in I}$ are said to be \textit{free}.

Let $\mathbb{C}_+$ and $\mathbb{C}_-$ denote the upper and lower half-planes, respectively. 
The \textit{Cauchy transform} of a probability measure $\mu$ on $\mathbb{R}$ is defined, for
$z\in \mathbb{C}\backslash \mathbb{R}$, by
$$G_{\mu }(z)=\int_{\mathbb{R}}\frac{1}{z-x}\mu \left( \mathrm{d}x\right).$$
It is well known that $G_{\mu }:\mathbb{C}_{+}\to \mathbb{C}_{-}$ is analytic and that $G_{\mu }$ determines uniquely the
measure $\mu$. 
The \textit{reciprocal Cauchy transform} is the function
$F_{\mu}:\mathbb{C}_{+}\rightarrow \mathbb{C}_{+}$ defined by $F_{\mu }\left(
z\right) =1/G_{\mu }(z).$ It was proved in \cite{BeVo93} that there are positive numbers $\alpha $
and $M$ such that $F_{\mu }$ has a right inverse $F_{\mu }^{-1}$ defined on the region
$$
\Gamma _{\alpha ,M}:=\left\{ z\in \mathbb{C};\left\vert \Re(z)\right\vert <\alpha\Im(z),\,\, 
\Im(z)>M\right\} . 
$$
The \textit{Voiculescu transform of }$\mu $ is defined by
$$
\phi _{\mu }(z)=F_{\mu }^{-1}(z)-z  \label{VoTr}
$$
on any region of the form $\Gamma _{\alpha ,M}$ where $F_{\mu }^{-1}$ is defined, see \cite{BeVo93}. The \textit{free cumulant transform} is a variant of $\phi _{\mu }$
defined as
	$$
	\mathcal{C}_{\mu }^{\boxplus}(z)=z\phi _{\mu }\left(\frac{1}{z} \right)
	=zF_{\mu }^{-1}\left(\frac{1}{z}\right) -1,  
	$$
for $z \in D_{\mu }:=\{z \in \mathbb{C}_-: z^{-1} \in \Gamma _{\alpha,M}\}$, see \cite{BNT06}.

The \textit{free additive convolution} $\mu_1 \boxplus \mu_2$ of two probability measures $\mu _{1},\mu _{2}$ on $ \mathbb{R}$ is defined so that 
$\phi _{\mu _{1}\boxplus \mu _{2}}\mathcal{(}z)=\phi _{\mu _{1}}
\mathcal{(}z)+\phi _{\mu _{2}}\mathcal{(}z)$, or equivalently, 
	$
	\mathcal{C}_{\mu _{1}\boxplus \mu _{2}}^{\boxplus }\mathcal{(}z)
	=\mathcal{C} _{\mu_{1}}^{\boxplus }\mathcal{(}z)
	+\mathcal{C}_{\mu _{2}}^{\boxplus }\mathcal{(}z)
	$
for $z\in D_{\mu _{1}}\cap D_{\mu _{2}}$. The measure $\mu _{1}\boxplus \mu _{2}$ is the
distribution of the sum $X_{1}+X_{2}$ of two free random variables $X_{1}$ and $X_{2}$ having
distributions $\mu _{1}$ and $\mu _{2}$ respectively.
 
The\textit{\ free multiplicative convolution} $\mu_1\boxtimes \mu_2$ of probability measures $\mu_1,\mu_2 \in \mathcal{M}$, one of them in $\mathcal{M}^+$, say $\mu_1 \in \mathcal{M}^+$, is defined as the distribution of $\mu _{X_{1}^{1/2}X_{2}X_{1}^{1/2}}$ where 
$X_1 \geq 0, X_2$ are free, self-adjoint elements such that $\mu _{X_{i}}=\mu _{i}$. The element $ X_{1}^{1/2}X_{2}X_{1}^{1/2}$ is self-adjoint and its distribution depends only on $\mu _{1}$ and $\mu _{2}.$ The operation $\boxtimes$ on $\mathcal{M}^+$ is associative and commutative.

The next result was proved in \cite{BeVo93}. 

\begin{proposition}
\label{PisiForPos} Let $\mu\in\mathcal{M}^{+}$ such that $\mu(\{0\})<1$. The function
	$
	\Psi_{\mu}(z)=\int_{0}^{\infty}\frac{zx}{1-zx}\mu(\mathrm{d}x)
	$
defined in $\mathbb{C}\backslash\mathbb{R}_{+}$ is univalent in the left-plane $i\mathbb{C}_{+}$ and 
$\Psi_{\mu}(i\mathbb{C} _{+})$ is a region contained in the circle with diameter $(\mu(\{0\})-1,0)$. Moreover,
$\Psi_{\mu}(i\mathbb{C}_{+})\cap\mathbb{R}=(\mu(\{0\})-1,0)$.
\end{proposition}

Let $\chi _{\mu }:$ $\Psi _{\mu }(i\mathbb{C}_{+})$ $\rightarrow i\mathbb{C} _{+}$ be the
inverse function of $\Psi _{\mu }.$ The $S$-\textit{transform} of $\mu $ is the function
$
S_{\mu }(z)=\chi (z)\frac{1+z}{z}\text{.}
$
The $S$-transform is an analytic tool for computing free multiplicative convolutions. The following was first shown in \cite{Vo87b} for measures in $\mathcal{M}^{+}$ with
bounded support, and then extended to measures in $\mathcal{M}^{+}$ with unbounded support~\cite{BeVo93},   measures in $\mathcal{M}$ with compact support~\cite{RaSp07}  and symmetric measures~\cite{APA09}. 

\begin{proposition}
\label{StransfPos} Let $\mu _{1} \in \mathcal{M}^{+}$ and $\mu _{2}$ a probability measure in $\mathcal{M}^{+}$ or symmetric, with $\mu _{i}\neq \delta _{0}$, $i=1,2.$ Then $\mu _{1}\boxtimes $ $\mu _{2}\neq \delta _{0}$
and
\begin{equation*}
S_{\mu _{1}\boxtimes \mu _{2}}(z)=S_{\mu _{1}}(z)S_{\mu _{2}}(z)
\end{equation*}
in the common domain containing $(-\varepsilon ,0)$ for small $\varepsilon
>0.$ Moreover, $(\mu _{1}\boxtimes $ $\mu _{2})(\{0\})=\max \{\mu _{1}(\{0\}),\mu
_{2}(\{0\})\}.$
\end{proposition}

Using this $S$-transform it was proved in \cite{APA09} that, for a $\mu\in\mathcal{M}^+$ and $\nu$ a symmetric probability measure, the following relation holds: 

\begin{equation}\label{square product}
(\mu\boxtimes\nu)^2=\mu\boxtimes\mu\boxtimes\nu^2
\end{equation} 
where, for a measure $\mu$, we denote by $\mu^2$ the measure induced by the push-forward $t\rightarrow t^2$.

\subsection {Free infinite divisibility}

\begin{definition}
Let $\mu$ be a probability measure in $\mathbb{R}$. We say that $\mu$ is \textbf{freely (or $\boxplus$- for short) infinitely divisible}, if for all $n $, there exists a probability measure $\mu _{n}$ such that 
\begin{equation}
\mu =\underset{n\text{ times}}{\underbrace{\mu _{n}\boxplus \mu _{n}\boxplus....\boxplus \mu _{n}}}\text{.}  \label{InfDivLib}
\end{equation}%
We denote by $I^\boxplus $ the class of such measures. 
\end{definition}
For $\mu \in I^\boxplus$, a free  convolution semigroup $(\mu^{\boxplus t})_{t\geq 0}$ can always be defined so that 
$\mathcal{C}^\boxplus_{\mu^{\boxplus t}}(z) = t \mathcal{C}^\boxplus_\mu(z)$. 

Now, recall that a probability measure $\mu$ is classically infinitely divisible if and only if its classical cumulant transform $\mathcal{C}_\mu^\ast(u) :=\log\left(\int_{\mathbb{R}}e^{iux}\mu(dx)\right)$ has the L\'{e}vy-Khintchine representation
\begin{equation}
\mathcal{C}_\mu^\ast(u)=i\eta u-\frac{1}{2}au^{2}+\int_{
\mathbb{R}}(e^{iut}-1-iut1_{\left[ -1,1\right] }\left( t\right) )\nu \left( dt\right),
\text{ \ \ }u\in 
\mathbb{R},  \label{levykintchine clasica}
\end{equation}
where $\eta \in \mathbb{R},$ $a\geq 0$ and $\nu $ is a L\'{e}vy measure on $\mathbb{R}$, that is,
$\int_{\mathbb{R}}\min (1,t^{2})\nu (dt)<\infty $ and $\nu (\{0\})=0$. If this representation exists, the
triplet $(\eta ,a,\nu )$ is unique and is called the classical characteristic triplet of $\mu $.

A $\boxplus$-infinitely divisible measure has a free analogue of the L\'{e}vy-Khintchine representation (see \cite{BNT06}). 

\begin{proposition}
\label{Cum-Ole}A probability measure $\mu $ on $%
\mathbb{R}$ is $\boxplus $-infinitely divisible if and only if there are $\eta \in 
\mathbb{R},$ $a\geq 0$ and a L\'{e}vy measure $\nu$ on $\mathbb{R}$ such that 
\begin{equation}
\mathcal{C}^\boxplus_{\mu }\mathcal{(}z)=\eta z+az^{2}+\int_{\mathbb{R}}\left( \frac{1}{1-zt}-1-tz1_{\left[ -1,1\right] }\left( t\right) \right) \nu \left(
dt\right) ,\quad z\in \mathbb{C}_{-}\text{.}  \label{levykintchine libre}
\end{equation}
The triplet $(\eta ,a,\nu )$ is unique and is called the free characteristic triplet of $\mu .$
\end{proposition}

The expressions (\ref{levykintchine clasica}) and (\ref{levykintchine libre}) give a natural bijection between 
$I^\ast $ and $I^\boxplus$. This bijection was introduced by Bercovici and Pata
 in \cite{Be-Pa99} in their studies of domains of attraction in free probability. Explicitly, this bijection is given as follows.

\begin{definition}
By the \textbf {Bercovici-Pata bijection} we mean the mapping $\Lambda :I^\ast \rightarrow I^\boxplus$ that sends the measure $\mu $
in $I^\ast $ with classical characteristic triplet $(\eta ,a,\nu )$
to the measure $\Lambda (\mu )$ in $I^\boxplus$ 
with free characteristic triplet $(\eta ,a,\nu )$.
\end{definition}

The map $\Lambda (\mu )$ is both a homomorphism in the sense that $\Lambda (\mu\ast\nu)=\Lambda (\mu)\boxplus\Lambda (\nu)$, and a homeomorphism with respect to weak convergence.

Another type of L\'evy-Khintchine representation in terms of $\phi_{\mu}$ is sometimes more useful than the free cumulant case: for $\mu \in I^{\boxplus}$, there exists a unique $\gamma_{\mu}\in\real$ and a finite non-negative measure $\tau_{\mu}$ on $\real$ such that
	$$
	\phi_{\mu}(z) = \gamma_{\mu}
	+\int_{\real} \frac{1+xz}{z-x}\tau_{\mu}(dx).
	$$

Finally let us mention very well known $\boxplus$-infinitely divisible measures that we will use often in this paper. The first one is the standard Wigner semicircle law $w$ with density 
\begin{equation*}
\frac{1}{2\pi }(4-x^2)^{1/2}\mathrm{d}x,\quad -2<x<2. 
\end{equation*}
The second is the Marchenko-Pastur law $m$, also known as free Poisson, with density 
\begin{equation*}
\frac{1}{2\pi }x^{-1/2}(4-x)^{1/2}\mathrm{d}x,\quad 0<x<4.
\end{equation*}

\subsection{Boolean convolutions}
The additive boolean convolution  $\mu \uplus \nu$ of probability measures on $\mathbb{R}$ was introduced in \cite{S-W}. It is characterized by 
$K_{\mu \uplus \nu} (z)=K_\mu(z)+K_\nu(z)$, where $K_\mu(z)$ is the energy function~\cite{S-W} defined by 
\[
K_\mu(z)=z-F_\mu(z).  
\]
Any probability measure is infinitely divisible with respect to the boolean convolution 
and a kind of L\'{e}vy-Khintchine representation is written as~\cite{S-W}
\[
K_\mu(z) = \gamma_\mu +\int_{\real}\frac{1+xz}{z-x}\eta_\mu(dx),  
\]
where $\gamma_\mu \in \mathbb{R}$ and $\eta_\mu$ is a finite non-negative measure. A boolean convolution semigroup $(\mu^{\uplus t})_{t\geq 0}$ can always be defined for any probability measure $\mu \in \mathcal{M}$. Moreover, if $\mu\in\mathcal{M}^+$ then  $\mu^{\uplus t}\in\mathcal{M}^+$  for all $t>0$.
 The Bercovici-Pata bijection $\mathbb{B}$ from the boolean convolution to the free one can be defined in the same way as for $\Lambda$, by the relation $K_\mu=\phi_{\mathbb{B}(\mu)}$. The reader is referred to \cite{Be-Pa99} for the definition of $\mathbb{B}$ in terms of domains of attraction.  

Similarly to $\Lambda$, $\mathbb{B}$ is a homomorphism between $(\mathcal{M},\uplus)$ and $(I^\boxplus,\boxplus)$, in the sense that $\mathbb{B} (\mu\uplus\nu)=\mathbb{B}(\mu)\boxplus\mathbb{B}(\nu)$. Also, $\mathbb{B}$ is a homeomorphism with respect to weak convergence.

\section{Free regular measures}
Let us consider a probability measure $\sigma \in I^\boxplus$ whose L\'evy measure $\nu$ of (\ref{levykintchine libre}) satisfies $\int_{\mathbb{R}_+}\min(1,t)\nu(dt) < \infty$. Then the L\'evy-Khintchine representation reduces to 
\begin{equation}\label{eq00}
\calC^{\boxplus}_{\sigma }(z)=\eta' z+\int_{\mathbb{R}}\left( \frac{1%
}{1-zt}-1\right) \nu \left( dt\right) ,\quad z\in \mathbb{C}_{-},
\end{equation}
where $\eta' \in \real$. The measure $\sigma$ is said to be a \textbf{free regular infinitely
divisible (or free regular, for short) distribution} if $\eta' \geq 0$ and $\nu\left((-\infty,0]\right)=0$. 
The most typical example is some compound free Poisson distributions.  
If the drift term $\eta'$ is zero and the L\'evy measure $\nu$ is $\lambda \rho$ for some $\lambda>0$ and a probability measure $\rho$ on $\real$, then we call $\sigma$ a {\bf compound free Poisson distribution} with rate $\lambda$ and jump distribution $\rho$. To clarify these parameters, we denote $\sigma = \pi(\lambda,\rho)$. 
\begin{remark} \label{rem free poisson}
1) The Marchenko-Pastur law $m$ is a compound free Poisson with rate $1$ and jump
distribution $\delta _{1}$. 

2) The compound free Poisson $\pi (1,\nu)$ coincides with the free multiplication $m \boxtimes\nu$. 
\end{remark}
 
This section is devoted to clarify several characterizations of free regular measures, some of which can be inferred from results of \cite{BG08,Has10,PeSa} and \cite{Sa10}, as we recollect in the following theorem. The final characterization uses free L\'evy processes which we will describe in details. 
\begin{theorem}\label{charac}

The following conditions for $\mu \in \mathcal{M} $ are equivalent:
\begin{enumerate}[{\rm (i)}]

\item $\mu$ is free regular. 

\item $\mu \in \Lambda(\mathcal{M}^+\cap I^\ast)$. 

\item $\mu \in \mathbb{B}(\mathcal{M}^+)$. 

\item $\mu ^{\boxplus t}\in \mathcal{M}^{+}$ for any $t >0$. 

\item $\mu$ is $\boxplus$-infinitely divisible, $\tau _{\mu }(-\infty ,0)=0$ and $\phi_\mu (-0) \geq 0$, where $\tau_\mu$ is the measure appearing in the representation of the Voiculescu transform. 

\item There exists a free subordinator $X_t$ such that $X_1$ is distributed as $\mu$.
\end{enumerate}
\end{theorem} 
\subsection{Characterizations (ii)--(v)}

The equivalence between (i) and (ii) is clear from the L\'{e}vy-Khintchine representation. However, we remark again that not all non-negative $\boxplus$-infinitely divisible distributions are free regular; a typical example of a measure in $I^\boxplus \cap \mathcal{M}^{+}$ but not in $I_{r+}^{\boxplus }$ is $w_{+}$, a semicircle distribution with mean $2$ and variance $1$. 

In a similar fashion, one can prove the equivalence between (i) and (iii). This can be seen from the boolean L\'evy-Khintchine representation of $\mu \in \mathcal{M}^+$ in terms of $K_\mu$, see Proposition 2.5 of \cite{Has10} for the details. 

The equivalence between (i) and (iv) was proved by Benaych-Georges~\cite{BG08} as the following lemma, see also Sakuma~\cite{Sa10}. 
\begin{lemma}
\label{characterization regular}A probability measure $\mu $ is in $I_{r+}^{\boxplus }$, if and only if $\mu ^{\boxplus t}\in \mathcal{M}^{+}$ for all $t>0.$
\end{lemma}

The equivalence between (i) and (v) is proved as follows. For a measure $\nu$ we denote by $a(\nu)$ the \textit{left extremity} of $\nu$: $a(\nu)=\min \{x: x \in \text{supp } \nu \}$.
\begin{proposition}\label{prop345} Let $\mu$ be a $\boxplus$-infinitely divisible distribution. Then $\mu$ is free regular if and only if $a(\tau_\mu) \geq 0$ and $\phi_\mu(-0) \geq 0$.  
\end{proposition}
\begin{proof}
Denote by $\mathbb{B}$ the Bercovici-Pata bijection from boolean to free convolutions: $z - F_\mu(z) = \phi_{\mathbb{B}(\mu)}(z)$. Let us denote by $z - F_\mu(z) = \gamma_\mu +\int_{\mathbb{R}}\frac{1+xz}{z-x}\eta_\mu(dx)$ the boolean L\'{e}vy--Khintchine representation. As proved in Proposition 2.5 of \cite{Has10} 
$\text{supp~} \mu \subset [0,\infty)$ if and only if $\text{supp~} \eta_\mu \subset [0,\infty)$ and $F_\mu(-0) \leq 0$.  By definition, $\nu$ is free regular if and only if $\mathbb{B}^{-1}(\nu)$ is supported on $[0,\infty)$, yielding the conclusion. 
\end{proof}

As we saw,  $\mu \in I^{\boxplus}\cap \mathcal{M}^+$ does not imply $\mu \in I^\boxplus _{r+}$. However, if $\mu$ has a singularity at $0$, such an implication is possible. We need a lemma to prove it.  
\begin{lemma}\label{lem345} 
Let $\mu$ be a $\boxplus$-infinitely divisible distribution with $a(\mu) > -\infty$. Then $a(\tau_\mu) \geq F_\mu(a(\mu)-0)$. 
\end{lemma}
\begin{proof}
Since $F_\mu$ is strictly increasing in $(-\infty, a(\mu))$, one can define $F_\mu^{-1}$ in an open set of $\mathbb{C}$ containing $(-\infty, F_\mu(a(\mu)-0))$. This gives an analytic continuation of $F_\mu^{-1}$ from $\mathbb{C}\ \backslash \real$ to $\mathbb{C} \backslash [F_\mu(a(\mu)-0), \infty)$. 
Therefore, $\tau_\mu$ is supported on $[F_\mu(a(\mu)-0), \infty)$. 
\end{proof}
\begin{theorem}\label{thm110}
Let $\mu$ be a $\boxplus$-infinitely divisible measure supported on $[0,\infty)$ satisfying either of the following conditions: (i) $\mu(\{0 \}) >0$; (ii) $\mu(\{0 \}) =0$ and $\int_0 ^1 \frac{\mu(dx)}{x} = \infty$.  Then $\mu$ is free regular. 
\end{theorem}
\begin{proof}
By assumption, $F_\mu(-0) =0$. Lemma \ref{lem345} implies that $a(\tau_\mu) \geq 0$. Taking the limit $z \nearrow 0$ in the identity $\phi_\mu(F_\mu(z)) = z - F_\mu(z)$, one concludes that $\phi_\mu(-0) =0$. Therefore, $\mu$ is free regular from Proposition \ref{prop345}. 
\end{proof}

\subsection{Free subordinators and free regular measures}
A particularly important family of real-valued processes with independent increments
is that of L\'{e}vy processes~\cite{Bertoin, Sato}. Let us recall the definition of a L\'{e}vy process. A continuous-time process $\{X_t\}_{t\geq 0}$ 
with values in $\mathbb{R}$ is called a L\'{e}vy process if 
\begin{enumerate}[{\rm (1)}]
\item 
Its sample paths are right-continuous and have left limits at every time point $t$.
\item 
For all $0\leq t_1 < \cdots< t_n$, the random variables $X_{t_1}, X_{t_2}-X_{t_1},\cdots, X_{t_n}-X_{t_{n-1}}$ are independent.
\item
For all $0\leq s\leq t$, the increments $X_t-X_s$ and $X_{t-s}-X_0$ have the same distribution. 
\item For any $s\geq0$, $X_{s+t}\rightarrow X_s$ in probability, as $t\to0$, i.e. the distribution of $X_{s+t}-X_s$ converges weakly to
$\delta_0$, as $t\rightarrow 0$.
\end{enumerate}
We assume that $X_0=0$. 
Now, if we denote by $\mu_{t}$ the distribution of $X_t$, 
then these measures satisfy the property
\begin{equation}\label{99}
\mu_{s+t} = \mu_s \ast \mu_t
\end{equation}
for any $s,t \geq 0$. The relation between infinitely divisible distributions and L\'{e}vy processes can be stated in the following proposition.

\begin{proposition} 
If $\{X_t\}_{t\geq 0}$ is a L\'{e}vy process, then for each $t>0$ the random variable $X_t$
has an infinitely divisible distribution. Conversely, if $\mu$ is an infinitely divisible distribution
then there is a L\'{e}vy process such that $X_1$ has distribution $\mu$.
\end{proposition}

A \emph{subordinator} is a real-valued L\'{e}vy process with non-decreasing path, this class has been broadly studied \cite{Bertoin, CoTa, Sato}.
\begin{proposition}\label{subordinator} 
Let $\{X_t\}_{t\geq 0}$ be a L\'{e}vy process. The process $X_t$ is
a subordinator if and only if the distribution of $X_1$ is supported on the positive real line.
\end{proposition}

Now, following Biane~\cite{Biane}, we define a process with free additive increments to be a map $t\mapsto X_t$ 
from $\mathbb{R}_+$ to the set of self-adjoint elements affiliated to some
$W^{\ast }$-probability space $(A, \varphi)$ such that, for any $0 \leq t_1 < \cdots< t_n$,
the elements $X_{t_1}, X_{t_2}-X_{t_1},\cdots, X_{t_n}-X_{t_{n-1}} $ are free. 
We also require the weak continuity of the distributions. However, we do not require an analog of  property (1) of a classical L\'evy process since there is no sample path in the $W^\ast$-probability setting. 

To define a free (additive) L\'evy process, we need stationarity. As Biane proposed, there are two natural classes which
deserve to be called free L\'{e}vy processes, depending on whether we ask for time homogeneity of the distributions of the increments
or of the transition probabilities. Here, we will use the former since in this case the distributions of a process form a semi-group for the free additive convolution. 

\begin{definition} A free additive L\'{e}vy process is a map $t\mapsto X_t$ from $\mathbb{R}_+$ to the set of self-adjoint elements affiliated to some
$W^{\ast }$-probability space $(A, \varphi)$, such that:
\begin{enumerate}[{\rm (1)}]
\item
For all $t_1 < \cdots< t_n$ , the elements $X_{t_1}, X_{t_2}-X_{t_1},\cdots, X_{t_n}-X_{t_{n-1}} $ are free.
\item
For all $0\leq s\leq t$ the increments $X_t-X_s$ and $X_{t-s}-X_0$ have the same distribution.
\item
For any $s\geq0$ in, $X_{s+t}\rightarrow X_s$ in probability, as $t\rightarrow0$, i.e. the distribution of $X_{s+t}-X_s$ converges weakly to
$\delta_0$, as $t\rightarrow 0$. 
\end{enumerate}
We also assume that $X_0=0$. 
\end{definition}
If we denote by $\mu_t$ the distribution of $X_t$, these measures satisfy the analog of (\ref{99}): 
\[
\mu_{s+t} = \mu_s \boxplus \mu_t
\]
for $s,t \geq0$. 
\begin{definition}
A free additive L\'{e}vy process  $X_t$ is called a \emph{free subordinator} if for all $0<s<t$ the increment $X_t-X_s$ is positive.
\end{definition}

We state the analog of Proposition \ref{subordinator} which clarifies the role of free regular measures in terms of free L\'{e}vy processes: they correspond to free subordinators.

\begin{proposition}\label{freesubordinator}
Let $X_t$ be a free additive L\'{e}vy process.
The process $X_t$ is a free subordinator if and only if the distribution of $X_1$ is free regular.
\end{proposition}

\begin{proof}
If $X_t$ is a free subordinator, it is clear that the distribution $\mu_1$ of $X_1$ is free regular since $X_t-X_0=X_t$ is positive and then the distribution $\mu_t=\mu_1^{\boxplus t}$ is supported on $\mathbb{R}_+$. Lemma \ref{characterization regular} allows us to conclude.

 Conversely, suppose that the distribution $\mu=\mu_1$ of $X_1$ is free regular. We want to see that $X_t-X_s$ is positive. Since $X_t$ is  a free L\'{e}vy process it is stationary and then $X_t-X_s$ has the same distribution as $X_{t-s}$, which is $\mu^{\boxplus (t-s)}$ and then, by Lemma \ref{characterization regular}, supported on  $\mathbb{R}_+$, i.e. $X_{t-s}$ positive. 
\end{proof}

\section{Closure properties}

   The following property was proved by Belinschi and Nica \cite{BN08}. For $\mu\in\mathcal{M}$ and $\nu\in\mathcal{M}^+$,
 \begin{equation}\label{homo}
\mathbb{B}(\mu\boxtimes\nu)=\mathbb{B}(\mu)\boxtimes\mathbb{B}(\nu).
\end{equation}

This suggested strongly that if $\mu $ and $\nu $ are $\boxplus$-infinitely
divisible then $\mu \boxtimes \nu $ is also $\boxplus$-infinitely divisible. However, this is not true in general, even
if both $\mu $ and $\nu$ belong to $ \mathcal{M}^{+}$ or $\mu=\nu$. The following counterexample was
given by Sakuma in \cite{Sa10}.

\begin{proposition} \label{sakuma ex}
Let $w_{+}$ be the Wigner distribution with density
\begin{equation*}
w_{2,1}(x)=\frac{1}{2\pi }\sqrt{4-(x-2)^{2}}\cdot 1_{[0,4]}(x)dx. 
\end{equation*}
Then $w_{+}\boxtimes w_{+}$ is not $\boxplus$-infinitely divisible.
\end{proposition}

It is not a coincidence that in this counterexample $w_{+}$  is not free regular. Indeed, if either $\nu$ or $\mu$ is free regular the problem is fixed.
\begin{proposition}
\label{closure} Let $\mu \in I_{r+}^\boxplus $ and $\nu \in I^{\boxplus },$
then $\mu \boxtimes \nu $ is freely infinitely divisible. Moreover if $\nu \in
I_{r+}^{\boxplus }$ then $\mu \boxtimes \nu \in I_{r+}^{\boxplus }.$
\end{proposition}

\begin{proof}

If $\mu\in I_{r+}^{\boxplus }$ then $\mu=\mathbb{B}(\mu_0)$ for some $\mu_0\in \mathcal{M}^+$. Similarly if $\nu\in I^{\boxplus }$ then $\mu=\mathbb{B}(\nu_0)$ for some $\nu_0\in \mathcal{M}$.
Then $\mu_0\boxtimes\nu_0$ is a well defined probability measure and Equation (\ref{homo}) gives $\mu\boxtimes\nu=\mathbb{B}( \mu_0\boxtimes\nu_0)\in I^{\boxplus }$. 

Now,  if $\nu\in I_{r+}^{\boxplus }$ then  $\nu_0\in \mathcal{M}^+$ and then $\mu_0\boxtimes\nu_0 \in \mathcal{M}^+$. Therefore $\mu\boxtimes\nu=\mathbb{B}( \mu_0\boxtimes\nu_0)\in I_{r+}^{\boxplus }$ since $\mathbb{B}$ sends positive measures to free regular ones. 
\end{proof}

As a consequence we answer a question in Sakuma and P\'{e}rez-Abreu \cite{PeSa}: If $\mu$ is free regular then $\mu\boxtimes\mu$ is also free regular. 

\begin{remark}
The previous proposition raises a question on a relation between the free subordinators associated to $\nu$, $\mu$ and $\nu\boxtimes\mu$. Let 
$D_a$ be the dilation operator defined by $\int_{\mathbb{R}}f(x)(D_a \mu)(dx) = \int_{\mathbb{R}}f(ax)\mu(dx)$ for any bounded continuous function $f$ and measure $\mu$. Equivalently, if a random variable $X$ follows a distribution $\mu$, $D_a \mu$ is the distribution of $aX$.  
For $\mu,\nu\in I^\boxplus_{r+}$, the identity 
 \begin{equation}
D_{1/t}(\mu ^{\boxplus t}\boxtimes \nu ^{\boxplus t})=(\mu \boxtimes \nu
)^{\boxplus t},\text{ }t\geq 0  \label{dilation}
\end{equation}
was essentially proved in \cite[Proposition 3.5]{BN08}. 
 This can be interpreted as follows in terms of processes. Let $X_t$ and $Y_t$ be free subordinators with $X_1\sim\mu$ and $Y_1\sim\nu$, which are free between them. Then the process $\frac{1}{t}X^{1/2}_tY_tX^{1/2}_t$ is distributed as a free subordinator  $Z_t$ such that $Z_1\sim\mu \boxtimes \nu$.
\end{remark}



It is clear from Proposition \ref{closure} that if $\mu $ is in $I_{r+}^{\boxplus }$ then
 $\mu ^{\boxtimes n}$ also belongs to $I_{r+}^{\boxplus },$ for $n\in \mathbb{N}$.
 Furthermore, this is also true for $t\geq 1,$ $\mu ^{\boxtimes t} \in I^{\boxplus}_{r+}$, when $t$
is not necessarily an integer, as we state in following proposition.

\begin{proposition}
If $\mu \in I_{r+}^{\boxplus }$, then for all $s\geq 1$, $\mu ^{\boxtimes
s}\in I_{r+}^{\boxplus }$.
\end{proposition}

\begin{proof}
By Lemma \ref{characterization regular}, it is enough to see that $(\mu ^{\boxtimes s})^{\boxplus t}\in \mathcal{M}^{+}$ for all $t>0$. For this, we use the following identity, essentially proved in \cite[Proposition 3.5]{BN08}:
\begin{equation}\label{1418}
D_{t^{s-1}}((\mu ^{\boxtimes s})^{\boxplus t})=(\mu ^{\boxplus
t})^{\boxtimes s}.
\end{equation}
Now, since $\mu $ is free regular, $\mu ^{\boxplus t}\in \mathcal{M}^{+}$
and then $(\mu ^{\boxplus t})^{\boxtimes s}\in \mathcal{M}^{+}$. Therefore, the RHS of Equation (\ref{1418})
defines a probability measure with positive support and then $(\mu ^{\boxtimes s})^{\boxplus t}\in \mathcal{M}^{+}$, as desired.
\end{proof}
Also, boolean powers less than one preserve free regularity.
\begin{proposition}
If $\mu \in I_{r+}^{\boxplus }$, then $\mu ^{\uplus t}\in I_{r+}^{\boxplus
} $ for $0\leq t\leq 1.$
\end{proposition}

\begin{proof}
It is shown in \cite{ArHa1} that if $\mu $ is $\boxplus$-infinitely divisible
then, for $0 < t < 1$,  
\begin{equation*}
\mathbb{B}((\mu ^{\boxplus (1-t)})^{\uplus t/(1-t)})=\mu ^{\uplus t}.
\end{equation*}%
Since $\mu $ is free regular, $\mu ^{\boxplus (1-t)}$ has a positive support and then,
since boolean convolution preserves measures with positive support, $
(\mu ^{\boxplus (1-t)})^{\uplus t/(1-t)}\in \mathcal{M}^{+}$. On
the other hand $\mathbb{B}$ sends positive measures to free regular
measures. 
\end{proof}

Finally we show that $I_{r+}^{\boxplus }$ is closed under convergence in distribution.

\begin{proposition}
\label{ConvRegular}Let $\left( \mu _{n}\right) _{n>0}$ be a sequence of
measures in $I_{r+}^{\boxplus }$. Suppose that $\mu _{n}{\rightarrow }\mu $
 for some measure $\mu .$ Then $\mu $ is also in $I_{r+}^{\boxplus }$. 
\end{proposition}

\begin{proof}
Let $\mu_n$ be a sequence of measures in $I^\boxplus_{r+}$ converging to $\mu$. Then, for each $n\in\mathbb{N}$, $\mu_n=\mathbb{B}(\nu_n)$ for some $\nu_n$ in $\mathcal{M}^+$. Since $\mathbb{B}$ is a homeomorphism $\nu_n\rightarrow\nu$, with $\nu$ some probability measure in $\mathcal{M}^+$ and satisfies that $\mathbb{B}(\nu)=\mu$. Hence $\mu\in I_{r+}^{\boxplus }$, as desired.
\end{proof}

\section{Squares of random variables with symmetric distributions in $I^\boxplus$}\label{sym}
We will prove Theorem \ref{main3} in this section. 
Given a probability measure $\mu$, we recall that $\mu ^{p}$ for $p \geq 0$ denotes the probability measure in $\mathcal{M}^{+}$ induced by the map $x\mapsto |x|^{p}$. For a measure $\lambda$ on $\real$ we denote by $\text{Sym}(\lambda)$ the symmetric measure $\frac{1}{2}\left(\lambda(dx) + \lambda(-dx)\right)$.



We quote a result from Sakuma and P\'erez-Abreu \cite[Theorem 12]{PeSa}. 
\begin{theorem}\label{thm123}
A symmetric probability measure $\mu$ is $\boxplus$-infinitely divisible if and only if there is a free regular distribution $\sigma$ such that 
$\calC^{\boxplus}_\mu(z) = \calC^{\boxplus}_\sigma(z^2)$. Moreover, the free characteristic triplets $(0, a_\mu,\nu_\mu)$ and $(\eta_\sigma,0,\nu_\sigma)$ are related as follows: 
$\nu_\mu = \text{\normalfont Sym}(\nu^{1/2}_\sigma)$ {\rm(}or equivalently $\nu_\sigma = \nu_\mu ^2$ {\rm )},  $a_\mu = \eta_\sigma$. 
\end{theorem}

The following proposition implies that the square of a symmetric measure which is $\boxplus$-infinitely divisible is also $\boxplus$-infinitely divisible. A similar result is proved for the rectangular free convolution of Benaych-Georges~\cite{BG08}.
\begin{proposition}\label{main} Let $\mu$ be a $\boxplus$-infinitely divisible  symmetric measure then $\mu^2 = m \boxtimes \sigma$, the compound free Poisson with rate $1$ and jump distribution $\sigma$, where $\sigma$ is the free regular distribution of Theorem \ref{thm123}. Conversely, if $\sigma$ is free regular, then $\text{\normalfont Sym}\left((m \boxtimes \sigma)^{1/2}\right)$ is $\boxplus$-infinitely divisible. 
\end{proposition}
\begin{proof}
We prove that the following are equivalent: 

(a) $\mu^2 = m\boxtimes \sigma$, 

(b) $\calC^{\boxplus}_\mu (z) = \calC^{\boxplus}_\sigma(z^2)$. 

\noindent
Indeed,  if $\mu^2 = m \boxtimes \sigma$, then $S_{\mu^2}(z) = S_m(z) S_\sigma(z) = \frac{1}{1+z}S_\sigma(z)$. 
Combined with the relation $S_{\mu^2}(z)  = \frac{z}{1+z}S_\mu(z)^2$, this implies $zS_\sigma(z) = (zS_\mu(z))^2$. 
Since the inverse of $zS_\lambda(z)$ is equal to $\calC^{\boxplus}_\lambda$ for a probability measure $\lambda$, we conclude that 
$(\calC^{\boxplus}_{\sigma})^{-1}(z) = ((\calC^{\boxplus}_{\mu})^{-1}(z))^2$, which is equivalent to (b). Clearly the converse is also true.
The desired result immediately follows from the above equivalence and Theorem \ref{thm123}. 
\end{proof}
This completes the proof of Theorem \ref{main3}(1). The result (2) for compound free Poissons is a consequence of Theorem \ref{thm123}.  


Now the following result of Sakuma and P\'erez-Abreu \cite[Theorem 22]{PeSa} follows as a consequence of Theorem \ref{main3}.
 
\begin{theorem}\label{PeSa2} Let $\sigma\in\mathcal{M}^+$ and $w$ be the standard semicircle law. Then $\sigma\boxtimes\sigma \in I^\boxplus_{r+}$ if and only if $\mu=w\boxtimes\sigma \in I^\boxplus$.
\end{theorem}

\begin{remark}
 It is not true that the square of a symmetric infinitely divisible
distribution in the classical sense is also infinitely divisible. For
instance, if $N_{1}$ and $N_{2}$ are independent Poissons then $SN=N_{1}-N_{2}$ is also
infinitely divisible and $(SN)^{2}$ is not infinitely divisible since the
support of $(SN)^{2}$ is $\{0,1,4,9,25...\}.$ (See \cite[pp. 51.]{StVH})
\end{remark}

There are two interesting consequences of Proposition \ref{main}. First, Proposition \ref{main} allows us to identify some non trivial free
regular measures which are in $I^{\ast}\cap I^{\boxplus}$: $ \chi^2$ and $ F(1,1)$. This will be explained in example \ref{chi}. 

The second consequence is on the commutator of two free even elements, which was pointed out to us by Speicher. 
See A.2 in the Appendix for the definition of even elements. In this case, an even element simply means that its distribution is symmetric. 

\begin{corollary}\label{commutator} Let $a_1,a_2$ be free, self-adjoint and even elements whose distributions  $\mu_1,\mu_2$ are $\boxplus$-infinitely divisible. Then the distribution of the free commutator $\mu_1\Square\mu_2:=\mu_{i(a_1a_2-a_2a_1)}$ is also $\boxplus$-infinitely divisible.
\end{corollary}
\begin{remark}
If $a_1, a_2$ are free, even and self-adjoint, the distribution of the anti-commutator $\mu_{a_1a_2+a_2a_1}$ is the same as $\mu_{i(a_1a_2-a_2a_1)}$~\cite{NiSp98}. 
\end{remark}
\begin{proof}
It was proved by Nica and Speicher \cite{NiSp98} that  $\mu_1\Square\mu_2$ is also symmetric and satisfies 
\begin{equation}\label{comm}
((\mu_1\Square\mu_2)^{\boxplus 1/2})^2=\mu^2_1 \boxtimes \mu^2_2. 
\end{equation}
Since, for $i=1,2$, the distribution $\mu_i$ is symmetric and belongs to $I^\boxplus$, by Proposition \ref{main}, we have the representation $\mu_i^2=m\boxtimes\sigma_i$, for some $\sigma_i$ free regular.  Then $((\mu_1\Square\mu_1)^{\boxplus 1/2})^2=m\boxtimes\sigma$ with $\sigma=m\boxtimes \sigma_1\boxtimes \sigma_2$. Now, by Theorem \ref{closure}, $\sigma$ is free regular and then $(\mu_1\Square\mu_2)^{\boxplus 1/2}$ is $\boxplus$-infinitely divisible. The desired result now follows.
\end{proof}

When we restrict $\mu_1$ to the standard semicircle law, we obtain the analog of Theorem \ref{PeSa2} for the free commutator.
\begin{corollary} Let $\sigma$ be a symmetric measure and $w$ be the standard semicircle law. Then $\sigma^2\in I^\boxplus_{r+}$ if and only if $\mu=w\square\sigma \in I^\boxplus$.
\end{corollary}

\begin{proof}
It is well known that the $w^2=m$ and then we get from Equation (\ref{comm}) that $((w\Square\sigma)^{\boxplus 1/2})^2=m\boxtimes \sigma^2$. The result now follows from Proposition \ref{main}.
\end{proof}

Moreover, Nica and Speicher reduced the problem of calculating the cumulants of the free commutator to symmetric measures. A further analysis of this reduction in combination with Corollary \ref{commutator} enables us to omit the assumption of evenness.  
\begin{theorem} \label{commutator 2}
Let $a_1$ and $a_2$ be free and self-adjoint elements, and let $\mu_1:=\mu_{a_1}$ and $\mu_2:=\mu_{a_2}$ be $\boxplus$-infinitely divisible distributions. Then the distribution of the free commutator $\mu_1\Square\mu_2:=\mu_{i(a_1a_2-a_2a_1)}$ is also $\boxplus$-infinitely divisible.
\end{theorem}
The proof uses combinatorial tools and will be given in the Appendix.

\begin{remark}[Polynomials on free variables]
So far we have proved that if $a_1,a_2,a_3$ are free even random variables
whose distributions are $\boxplus$-infinitely divisible, then $i(a_ia_j-a_ja_i)$, $a_ia_j+a_ja_i$ and $a_i^2$ also have $\boxplus$-infinitely divisible distributions  
(for the free commutator, the assumption of evenness is not needed). Combining these results one can easily see that the following polynomials are also $\boxplus$-infinitely divisible:
$a_1^2+a_2^2+a_2a_1+a_2a_1$, $i(a_1a_2^2-a_2a_1^2)$, 
$a_1^4+a_2^4-a_2^2a_1^2-a_2^2a_1^2$,
$a_1a_2^2a_1+a_2a_1^2a_2+a_1a_2a_1a_2+a_2a_1a_2a_1$,
$a_1a_2^2a_1+a_2a_1^2a_2-a_1a_2a_1a_2-a_2a_1a_2a_1$,
$a_1a_2a_3+a_2a_1a_3+a_3a_1a_2+a_3a_2a_1$, etc. Therefore, it is natural to ask for which polynomials free 
infinite divisibility is
preserved.
\end{remark}

\section{Examples, conjectures and future problems}

In this section, we gather some examples related to our results.
From these examples, we also present open problems.

As a first example we use Theorem \ref{main3} to identify measures in $I^{\ast}\cap I^{\boxplus}_{r+}$.

\begin{example} \label{chi}
The following are measures which are both classically and
freely infinitely divisible.
\begin{enumerate}[{\rm (1)}]
\item
Let $\chi ^{2}$ be a chi-squared with $1$ degree of freedom with density $$f(x):=\frac{1}{\sqrt{2\pi x}}e^{-x/2},~~~~x>0.$$
It is well known that $\chi ^{2}$ is infinitely divisible in the classical sense.
It was proved in \cite{BBSL09}\ that a symmetric Gaussian $Z$ is $\boxplus$-infinitely divisible. 
Hence, by Theorem \ref{main3}, $Z^{2}$ is free regular. $Z^{2}\sim\chi ^{2}$ and then $\chi ^{2}\in I^{\ast}\cap I^{\boxplus}_{r+}$

\item
Let $F(1,n)$ be an $F$-distribution with density $$f(x):=\frac{1}{B(1/2,n/2)}\frac{1}{(nx)^{1/2}} \left( 1+\frac{x}{n} \right)^{-(1+n)/2},\quad x>0.$$ $F(1,n)$ is classically infinitely divisible, as can be seen in \cite{IsKe}.
On the other hand $F(1,n)$ is the square of a t-student with n degrees of freedom $t(n)$. In particular $t(1)$ is the Cauchy distribution, hence by 
Theorem \ref{main3}, $F(1,1)$  belongs to $I^{\ast}\cap I^{\boxplus}_{r+}$.
\end{enumerate}
\end{example}

\begin{remark}
Numeric computations of free cumulants have shown that the chi-squared with $2$ degrees of freedom is not freely infinitely divisible. However, the free infinite divisibility of $t$-student with $n$ degrees of freedom is still an open question.
\end{remark}

Next, we give some examples of free regular measure from known distributions in non-commutative probability.
\begin{example}\label{EX 1}
\begin{enumerate}[{\rm (1)}]
\item 
Free one-sided stable distributions with non-negative drifts.
These distributions are found by Biane in Appendix in \cite{Be-Pa99}. 
\item
The square of a symmetric $\boxplus$-stable law. By Theorem \ref{main3} it is free regular, and moreover, by the results of \cite{APA09} we can identify the L\'{e}vy measure $\sigma$ of Theorem \ref{main3} with a $\boxplus$-stable law. Indeed, any symmetric stable measure has the representation $w\boxtimes\nu_{\frac{1}{1+t}}$ and then 
by Equation \eqref{square product} the square is $w^2\boxtimes \nu_{\frac{1}{1+t}}\boxtimes \nu_{\frac{1}{1+t}}=m\boxtimes \nu_{\frac{1}{1+2t}}$.
\item
Free multiplicative, free additive and boolean powers of the free Poisson $m$.
In particular, for $t\geq 1$the free Bessel laws $m^{\boxtimes t} \boxtimes m^{\boxplus s}$ studied in \cite{BBCC} are free regular.
\item
The free Meixner laws, which are introduced by 
Saitoh and Yoshida~\cite{S-Y} and Anshelevich \cite{An03},
whose L\'evy measures are given by
$$
\nu_{a,b,c}(dx) =c\frac{\sqrt{4b-(x-a)^{2}}}{\pi x^{2}}1_{a-2\sqrt{b}<x<a+2\sqrt{b}}(x) dx. 
$$
If $a-2\sqrt{b}\geq 0$, then the L\'evy measure is concentrated on $[0,\infty)$ and $\int_{\real} \min (1,|x|) \nu_{a,b,c}(dx)<\infty$. Thus, if the drift term is non-negative, then it will be free regular. 
This case includes the free gamma laws, which come from interpretation by orthogonal polynomials not the Bercovici-Pata bijection.
\item
The beta distribution $B(1-a,1+a)$ ($0 < a < 1$) has the density 
\[
p_a(x) = \frac{\sin(\pi a)}{\pi a} x^{-a}(1-x)^{a},~~0 < x < 1. 
\]
$B(1-a,1+a)$ is $\boxplus$-infinitely divisible if and only if $\frac{1}{2} \leq a < 1$~\cite{ArHa2}. Moreover, $B(1-a,1+a)$ is free regular for $\frac{1}{2} \leq a < 1$ since $\int_{0}^1 \frac{p_a(x)}{x}dx = \infty$ (see Theorem \ref{thm110}).  We note that $B(\frac{1}{2},\frac{3}{2})$ coincides with the Marchenko-Pastur law. 
\end{enumerate}
\end{example}

\begin{example}\label{ex w4}
Let $w$ be the standard semicircle law. Then $w^{2}$ and $w^{4}$ are both free
regular. It is well known that $w^{2}=m ,$ which is free regular. From \cite{ABNPA09},\ if $\mathbf{b}_{s}$
is the symmetric beta $(1/2, 3/2)$ distribution, $\mathbf{b}_{s}$ is freely
infinitely divisible and then, by Theorem \ref{main3}, $(\mathbf{%
b}_{s})^{2}$ is free regular. 

The symmetric beta distribution $\mathbf{b}_{s}$ has density
\begin{equation*}
\mathbf{b}_{s}(\mathrm{d}x)=\frac{1}{2\pi }\left\vert x\right\vert
^{-1/2}(2-\left\vert x\right\vert )^{1/2}\mathrm{d}x,\quad \left\vert
x\right\vert <2.
\end{equation*}
Clearly $m_{2n}\left( \mathbf{b}_{s}\right) =m_{4n}(w)$ and then $(\mathbf{
b}_{s})^{2}=w^{4}.$ Also since $w^{4}=m^{2}=(\mathbf{b}_{s})^{2},$ $w^{4}$
is free regular.
\end{example}

\begin{remark}
It is not known if $w^{2n}$ is $\boxplus$-infinitely divisible for all $n>0,$ as
in classical probability.
\end{remark}

One may ask if the example $w_+$ is an exception but the following example shows that there are a lot of measures in $I^\boxplus\cap\mathcal{M}^+$ which are not $I^\boxplus_{r+}$. We also mention here that a quarter-circle distribution is 
not $\boxplus$-infinitely divisible.  
\begin{example}
\begin{enumerate}[{\rm (1)}]
\item
We present a method to construct freely infinitely divisible measures with positive support, but not free regular. 
Let $\mu\neq\delta_0$ be $\boxplus$-infinitely divisible with compact support, say $[-a,b]$. Then $\mu_1:=\mu\boxplus \delta_a$ has support $[0,b+a]$ and $\mu_2:=\mu\boxplus \delta_{-b}$ has support $[-(a+b),0]$. Both $\mu_1$ and $\tilde{\mu}_2(dx):=\mu_2(-dx)$ are in $I^{\boxplus}\cap\mathcal{M}^+$, but either $\mu_1$ or $\tilde{\mu}_2$ must not be free regular.  

Indeed suppose that $\mu_1$ is free regular, then $\mu_1=\Lambda(\nu_1)$ for some $\nu_1\in\mathcal{M}^+$ with unbounded positive support, say $[c,\infty)$. Now, recall that $\Lambda$  is a homomorphism, so that $\mu_2=\Lambda(\nu_1*\delta_{-a-b})$ but since the support of $\nu$ is $[b,\infty)$ then the support of $\widetilde{\nu_1*\delta_{-a-b}}$ is $(-\infty,a+b-c]$ and intersects $\mathbb{R}_-$, which means that $\tilde{\mu}_2$ is not free regular.
In particular, if $\mu$ is symmetric, any shift of $\mu$ is not free regular. Easy explicit examples can also be obtained from $\mu$ a free regular measure, for instance from (3), (4) and (5) of Example \ref{EX 1}.

\item 
Let $a_\alpha$ be the monotone $\alpha$-stable law characterized by $F_{a_\alpha}(z) = (z^\alpha +e^{i\alpha \pi})^{1/\alpha}$, where the powers $z^{\alpha}$ and $z^{1/\alpha}$ are respectively defined as $e^{\alpha\log z}$ and $e^{\frac{1}{\alpha}\log z}$ in $\mathbb{C}\backslash [0,\infty)$. The function $\log$ is not the principal value, but is defined so that $\text{\normalfont Im}(\log z) \in (0,2\pi)$. If $\alpha \in [\frac{1}{2}, 1]$, this measure is $\boxplus$-infinitely divisible and supported on $[0,\infty)$~\cite{ArHa2,Biane}. However this measure is not free regular, since 
the Voiculescu transform $\phi_{a_\alpha}(z) = (z^\alpha -e^{i\alpha \pi})^{1/\alpha} -z$ is not analytic in $\mathbb{C}\backslash[0,\infty)$.  In this case the support of the L\'evy measure is $[-1,\infty)$.  
\item 
Let $\sigma>0$.
Suppose $q$ be the quarter-circle distribution, that is, it has density
\begin{align*}
f_{q}(x) =
\begin{cases}
\frac{1}{\pi\sigma^{2}} \sqrt{4\sigma^{2}-x^{2}} & \mathrm{(}x \in [0,2\sigma]\mathrm{)},\\
0 & (\text{otherwise}). 
\end{cases}
\end{align*}
It is not freely infinitely divisible for any $\sigma>0$. We can find it by the following proposition of free kurtosis.
\begin{proposition}
If $\mu$ is freely infinitely divisible then the free kurtosis $\mathrm{kurt}^{\boxplus}(\mu)$ of $\mu$ is positive, that is,
$$
\mathrm{kurt}^{\boxplus}(\mu) =\frac{\widetilde{m_{4}}(\mu)}{(\widetilde{m_{2}}(\mu))^{2}}-2>0,
$$
where $\widetilde{m_{2}}(\mu)$, $\widetilde{m_{4}}(\mu)$ are 2nd and 4th moments around mean.
\end{proposition}
For more detail of free kurtosis, see p.171 in \cite{APA10}. 
Here we can obtain moments of $q$ as follows:
$$
m_{1}(q) = \frac{8\sigma}{3\pi}, \,\,m_{2}(q) = \sigma^{2}, \,\,
m_{3}(q) = \frac{2^{6}\sigma^{3}}{15\pi}, \,\,m_{4}(q) = 2\sigma^{4}.
$$
Therefore,
$$
\frac{\left(2-\frac{2^{12}}{3^{3}\pi^{4}}\right)}{\left(1-\frac{2^{6}}{3^{2}\pi^{2}}\right)^{2}}-2<0
$$
for any $\sigma>0$.
In fact, this amount is around $-0.0233443$.
\end{enumerate}
\end{example}

Recall from Proposition \ref{sakuma ex} that $w_+\boxtimes w_+$ is not freely infinitely divisible. Therefore, we have the following conjecture.
\begin{conjecture}
If $\mu \in\mathcal{M}^+$ is $\boxplus$-infinitely divisible, then $\mu \boxtimes \mu $ is $\boxplus$-infinitely divisible if and only if $\mu $ is free regular. 
\end{conjecture}

\begin{example} [free commutators]
\begin{enumerate}[{\rm (1)}]
\item 
 Let $\sigma_{s}$ and $\sigma_{t}$ be two symmetric free stable distributions of index $s$ and $t$, respectively. Then by Corollary \ref{commutator} the free commutator $\sigma_s\square \sigma_t$ is $\boxplus$-infinitely divisible. For the case $t=s=2$ (the Wigner semicircle distribution) the density of $w\square w$ is given by \cite{NiSp98}
 \begin{equation} \label{com w}
f(t)=\frac{\sqrt{3}}{2\pi \mid t\mid}\left( \frac{3t^2+1}{9h(t)}-h(t)\right),~~~~\mid t\mid\leq\sqrt{(11+5\sqrt{5})/2}, 
\end{equation}
where 
$$h(t)=\sqrt[3]{\frac{18t^2+1}{27}+\sqrt{\frac{t^2(1+11t^2-t^4)}{27}}}.$$ 

\item 
Let $w$ be the standard semicircle law and let $\nu_{\frac{1}{1+2s}}$ be a positive free stable law, for some $s>0$. If we denote $\hat \nu_{\frac{1}{1+2s}}=\text{\normalfont Sym}(\nu_{\frac{1}{1+2s}}^{1/2})$ then $\mu:=w\square\hat\nu_{\frac{1}{1+2s}}$ is a symmetric free stable distribution with index $\frac{2}{1+2s}$. Indeed, by Equation (\ref{comm}), $\mu$ satisfies
 $$(\mu^{\boxplus 1/2})^2=((w\square\hat\nu_{\frac{1}{1+2s}})^{\boxplus 1/2})^2=w^2\boxtimes\nu_{\frac{1}{1+2s}}=m\boxtimes\nu_{\frac{1}{1+2s}}.$$
 From Equation (\ref{square product}) and results in \cite{APA09} we see that  $m\boxtimes\nu_{\frac{1}{1+2s}}=(w\boxtimes\nu_{\frac{1}{1+s}})^2$. This means that $\mu^{\boxplus 1/2}=w\boxtimes\nu_{\frac{1}{1+s}}$ which is a symmetric free stable distribution with index $\frac{2}{1+2s}$. The case $s=1/2$ was treated in  \cite[Example 1.14]{NiSp98}.

\item
Assume that $b$ is a symmetric Bernoulli distribution $\frac{1}{2}(\delta_{-1}+\delta_{1})$.
Let $\mu,\nu$ be symmetric distributions. Then the free commutator $\mu\square \nu$  is 2-$\boxplus$-divisible, but when $\mu=\nu$ we can identify $(\mu\square \mu)^{\boxplus 1/2}$. Indeed, by Eq.\ (\ref{comm}), $(\mu\square \mu)^{\boxplus 1/2}= \sqrt{\mu^2\boxtimes\mu^2}$.  On the other hand, by Equation (\ref{square product}), $(\mu^2\boxtimes b)^2=\mu^2\boxtimes\mu^2$. Hence $(\mu^2\boxtimes b)^{\boxplus 2}=\mu\square\mu$. 

When $\mu=w$ a strange thing happens: $w^2=m$, and $m\boxtimes b$ is a compound free Poisson with rate $1$ and jump distribution $b$, see Remark \ref{rem free poisson}. This implies that $w\square w=m\boxplus \widetilde{m}$, where $\widetilde{m}$ is defined by $\widetilde{m}(B) =m(-B)$.
It is a free symmetrization of the Poisson distribution (not to be confused with the symmetric beta of Example \ref{ex w4}). As pointed out in \cite{NiSp98}, this gives another derivation of the density of $w\square w$ given in Equation (\ref{com w}).

\item
For the free Poisson with mean $1$, the free commutator becomes $m\square m=(m \boxtimes m \boxtimes b)^{\boxplus 2}$, the compound free Poisson with rate $2$ and jump distribution $m \boxtimes b$. Indeed, if we define $\hat m := m \boxtimes b$, we have that $m\square m=\hat m\square \hat m$ since the even free cumulants of $\hat m$ are all one, the same as those of $m$, and since the free commutator of measures depends only on the even cumulants of the measures~\cite[Theorem 1.2]{NiSp98}. By Equation (\ref{square product}) we have $\hat m^2=m\boxtimes m$, and therefore  by Equation (\ref{comm}), we have $$( (m\square m)^{\boxplus 1/2})^2=m \boxtimes m \boxtimes m \boxtimes  m.$$ Again using  Equation (\ref{square product}) we see that $ m \boxtimes m \boxtimes m \boxtimes  m= (m \boxtimes  m \boxtimes  b)^2$. The claim then follows.
\end{enumerate}
\end{example}





\begin {appendix}

\section{Combinatorial approach}

In this appendix we prove Theorem \ref{commutator 2}, using combinatorial tools. We also give a combinatorial proof of Theorem \ref{main3} which was proved with analytic tools. We decided not to include them in the main section of this article not only because they are more involved but also since, in principle, these proofs are only valid when the existence of moments is assumed. However, we believe that a reader who is more acquainted with the combinatorial approach may find them more illuminating.

\subsection{Free cumulants}

A measure $\mu $ has \textit{all moments} if $m_{k}(\mu )=\int_{\mathbb{R}}t^{k}\mu (\mathrm{d}t)<\infty ,$ for each even integer $k\geq 1$. Probability measures with compact support have all moments.

The \textbf{free cumulants} $(\kappa_{n})$ were introduced by Voiculescu~\cite{Vo86} as an analogue of classical cumulants, and were developed more by Speicher~\cite{Sp94} in his
combinatorial approach to free probability theory. We refer the reader to the book of Nica and Speicher \cite{NiSp06} for a nice introduction to this combinatorial approach. Let $\mu$ be a probability measure with compact support, then the cumulants are the coefficients $\kappa_{n}=\kappa_{n}(\mu )$ in the series expansion
\begin{equation*}
\calC^{\boxplus}_{\mu }\mathcal{(}z)=\sum\nolimits_{n=1}^{\infty }\kappa_{n}(\mu )z^{n}.
\end{equation*}
 For a sequence $(t_n)_{n\geq 1}$ and a partition $\pi =\{V_{1},...,V_{r}\}\in NC(n) $  we denote $t_{\pi }:=t_{\left\vert V_{1}\right\vert }\cdots \kappa_{\left\vert V_{r}\right\vert }.$ 

The relation between the free cumulants and the moments is described by the
lattice of non-crossing partitions $NC(n)$, namely,
\begin{equation}
m_{n}(\mu )=\sum\limits_{\mathbf{\pi }\in NC(n)}\kappa_{\pi }(\mu ). \label{free moment-cumulant
formula}
\end{equation}
Since free cumulants are just the coefficients of the series expansion of  $\calC^{\boxplus}_{\mu }\mathcal{(}z)$, they linearize free convolution: 
\begin{equation*}
\kappa_{n}\mathcal{(}\mu _{1}\boxplus \mu _{2})=\kappa_{n}\mathcal{(}\mu _{1})+\kappa_{n} \mathcal{(}\mu_{2}). 
\end{equation*}

A compound free Poisson $\mu$ with rate $\lambda$ and jump distribution $\nu$ can be characterized as 
\begin{equation*}
\kappa_{n}(\mu )=\lambda m_{n}(\nu ).  
\end{equation*}%
In particular, if $\mu$ is of the form 
$m \boxtimes \sigma$ for a probability measure $\sigma$ on $\real$, then $\kappa_n(m\boxtimes \sigma)=m_n(\sigma)$. 

Compound free Poissons are $\boxplus$-infinitely divisible, and moreover, any $\boxplus$-infinitely divisible probability measure is a weak limit of compound free Poissons.

\subsection{Even elements}

When $\mu $ has all moments, being symmetric is equivalent to having
vanishing odd moments, that is $m_{2k+1}(\mu )=\int_{\mathbb{R}}t^{2k+1}\mu (%
\mathrm{d}t)=0.$ On the other hand $\mu ^{2}$ has moments $m_{k}(\mu^{2})=m_{2k}(\mu )$.

An element $x\in (\mathcal{A},\varphi )$ is said to be \textbf{even} if the only
non vanishing moments are even, i.e. $\varphi (x^{2k+1})=0$. Even
elements correspond to symmetric distributions. It is clear by the moment-cumulant formula (\ref{free moment-cumulant formula}) that $x\in \mathcal{A}$ is even if and only if the only non-vanishing free cumulants are even. 
In this case we call $(\alpha _{n}:=\kappa_{2n}(x))_{n\geq 1}$ the determining sequence of $x$.

The next proposition gives a formula for the cumulants of the square of an
even element in terms of the cumulants of this element and can be found in \cite[Proposition 11.25]{NiSp06}

\begin{proposition}
\label{Kum of x^2}Let $x\in \mathcal{A}$ be an even element and let $(\alpha_{n}=\kappa_{2n}(x))_{n\geq 1}$ be the determining sequence of $x$. Then the cumulants of $x^{2}\,$\ are given as follows:
\begin{equation*}
\kappa_{n}(x^{2})=\sum\limits_{\mathbf{\pi }\in NC(n)}\alpha_{\pi }.
\end{equation*}
\end{proposition}

Now we are able to prove the main result of this section.
\begin{proposition} \label{Sq Poisson}
Let $\mu $ be symmetric distribution with all moments. If $\mu $ is freely
infinitely divisible, then $\mu ^{2}$ is a compound free Poisson $\pi (\lambda ,\rho )$ with $\rho\in I^\boxplus_{r+}$. If
moreover $\mu $ is itself  a compound free  Poisson $\pi (\lambda ,\nu ),$
then $\rho $ is also a compound free  Poisson $\rho =\pi (\lambda ,\nu ^{2})$.
\end{proposition}

\begin{proof}
Let $x$ be an even element with distribution $\mu$ and suppose that $\mu$ is a symmetric compound free  Poisson with rate $\lambda $  and jump
distribution $\nu $ and let $\rho =\pi (\lambda ,\nu ^{2})$ be a 
compound free Poisson with rate $\lambda$ and jump distribution $\nu ^{2}$. Then
the determining sequence of $x$ is 
\begin{equation*}
\alpha _{n}=\kappa_{2n}(x)=\lambda m_{2n}(\nu )=\lambda m_{n}(\nu
^{2})=\kappa_{n}(\rho ). 
\end{equation*}%
By Proposition \ref{Kum of x^2} we have that 
\begin{equation*}
\kappa_{n}(x^{2})=\sum\limits_{\mathbf{\pi }\in NC(n)}\alpha
_{\pi }=\sum\limits_{\mathbf{\pi }\in NC(n)}\kappa_{\pi }(\rho)=m_{n}(\rho )
\end{equation*}%
and hence the distribution $\mu $ of $x^{2}$ is a compound free Poisson with
rate $1$ and jump distribution $\rho$. 

 More generally if $\mu \in I^\boxplus $ is symmetric, then $\mu$ can be
approximated by compound free Poissons which are symmetric, say $\mu
=\lim_{n\rightarrow \infty }\mu _{n}.$ By the previous case for each $ n>0$, $\mu _{n}^{2}=m\boxtimes\nu_n$
for some $\nu_n$ compound free Poisson, which is free regular. Since $\mu _{n}^{2}\rightarrow \mu ^{2}$ and $\nu_n \to \nu$ for some $\nu$, then $\mu=m\boxtimes\nu$. The measure $\nu$ is free regular since $I_{r+}^{\boxplus }$ is closed under the convergence in distribution.
\end{proof}

Finally, we use similar arguments to prove Theorem \ref{commutator 2} on free commutators.

\begin{proof}[Proof of Theorem \ref{commutator 2}]
By an approximation similar to Proposition \ref{Sq Poisson}, it is enough to consider $\mu_1$ and $\mu_2$ compound free  Poissons. Let $\mu_1\square\mu_2$ be the free commutator and  $\kappa_n(\mu_i)=\lambda_i m_n(\nu_i)$ the free cumulants of $\mu_i$, for $i=1,2$. It is clear that $m_{2n}(\nu_i)=m_{2n}(\text{\normalfont Sym}(\nu_i))$ and $m_{2n+1}(\text{\normalfont Sym}(\nu_i))=0$. 
 Now, by Theorem 1.2 in \cite{NiSp98}, the free cumulants of $\mu_1\square\mu_2$ only depend on the even free cumulants of $\mu_1$ and $\mu_2$, and therefore we can change $\mu_i$ by the symmetric compound Poisson with L\'{e}vy measure $\text{\normalfont Sym}(\nu_i)$. Thus by Corollary \ref{commutator} $\mu_1\square\mu_2$ is $\boxplus$-infinitely divisible as desired.
\end{proof}

\end{appendix}


\begin{thebibliography}{99}
\bibitem{An03}M. Anshelevich.
Free martingale polynomial.
\emph{J. Funct. Anal.}, 
\textbf{201}, (2003), 228--261.

\bibitem{APA09} O. Arizmendi and V. P\'{e}rez-Abreu.
The S-transform of symmetric probability measures with unbounded supports. 
\emph{Proc. Amer. Math. Soc}., 
\textbf{137}, (2009), 3057--3066.

\bibitem{APA10} O. Arizmendi and V. P\'{e}rez-Abreu.
On the non-classical infinite divisibility of power semicircle distributions. 
\emph{Commun. Stoch. Anal.},
\textbf{4}, (2010), 161--178. 

\bibitem{ArHa1} O. Arizmendi and T. Hasebe. 
Semigroups related to additive and multiplicative, free and Boolean convolutions. 
{\it preprint} (arXiv:1105.3344v3). 

\bibitem{ArHa2} O. Arizmendi and T. Hasebe. 
On a class of explicit Cauchy-Stieltjes transforms related to monotone stable and free Poisson laws.
{\it To appear in Bernoulli}.

\bibitem{ABNPA09} O. Arizmendi, O. E.  Barndorff-Nielsen and V. P\'{e}rez-Abreu.
On free and classical type $G$ distributions. 
Special number, \emph{\ Braz. J. Probab. Stat.}, {\bf 24}, (2010),
106--127. 

\bibitem{BBCC} T. Banica, S. T. Belinschi, M. Capitaine and B. Collins. 
Free Bessel Laws.
{\it Canad. J. Math.}, 
\textbf{63},
(2011), 3--37.

\bibitem{PABN}
O. E. Barndorff-Nielsen  and V. Perez-Abreu.
Matrix subordinators and related Upsilon transformations.
{\it Theory Probab. Appl.},
\textbf{52}, 
(2008), 
1--23.
\bibitem{BNT06} O. E. Barndorff-Nielsen  and S. Thorbj\o rnsen.
Classical and free infinite divisibility and L\'{e}vy processes.
In: \emph{Quantum Independent Increment Processes II}. M. Sch\"{u}rmann and
U. Franz (eds). 
Lecture Notes in Mathematics \textbf{1866}, Springer, Berlin (2006).

\bibitem{BBSL09} S. T. Belinschi, M. Bo\.zejko, F. Lehner and R. Speicher.
The normal distribution is $\boxplus$-infinitely divisible.
{\it Adv. Math.},
{\bf 226},
(2011),
3677--3698.

\bibitem{BN08} S. T. Belinschi and A. Nica.
On a remarkable semigroup of homomorphisms with respect to free multiplicative convolution. 
\emph{Indiana Univ. Math. J.}, 
\textbf{57},
(2008),
1679--1713.

\bibitem{BG08} F. Benaych-Georges.
On a surprising relation between the Marchenko-Pastur law, rectangular and square free convolutions. 
\emph{Annales de l'Institut Henri Poincar\'{e} (B) Probability and Statistics.}
{\bf 46},
(2010), 
644--652.

\bibitem{Be-Pa99} 
H. Bercovici and V. Pata.
Stable laws and domains of attraction in free probability theory 
(with an appendix by Philippe Biane).
{\it Ann. of Math. (2)},
\textbf{149},
(1999), 
1023--1060. 


\bibitem{BeVo93} H. Bercovici and D. Voiculescu.
Free convolution of measures with unbounded supports.
\emph{Indiana Univ. Math. J.}, 
\textbf{42},
(1993), 
733--773.

\bibitem{Be99}
J. Bertoin.
Subordinators: examples and applications. 
{\it Lecture Notes in Math}., 
\textbf{1717},
(1999),
1--91. 

\bibitem{Bertoin} J. Bertoin.
{\it L\'{e}vy Processes}.
Cambridge Tracts in Mathematics,
{\bf 121}, 
Cambridge University Press, 
(2002). 

\bibitem{Biane} Ph. Biane. 
Processes with free increments. 
{\it Math. Z.},
\textbf{227}, 
(1998), 
143--174.


\bibitem{CoTa} R. Cont and P. Tankov.
{\it Financial Modelling with Jump Processes}. 
Chapman Hall, (2003).
 
\bibitem{Fel} W. Feller. 
{\it An Introduction to Probability Theory and its Applications}. 
Cambridge University Press, 
(1987).


\bibitem{Has10} T. Hasebe.
Monotone convolution semigroups.
{\it Studia Math.},
{\bf 200}, 
(2010), 
175--199.


\bibitem{IsKe} M. E. H. Ismail and D. H. Kelker.
Special Functions, Stieltjes Transforms and Infinite Divisibility
 \emph{SIAM  J. Math. Anal.} 
{\bf 10}(5), (1979), 884--901.



\bibitem{Ml} W.\ M\l{o}tkowski.
 Combinatorial relation between free cumulants and Jacobi parameters, 
\emph{Infin. Dimens. Anal. Quantum Probab. Relat. Topics }
\textbf{12} (2009), 291--306.


\bibitem{NiSp98} A. Nica and R. Speicher.
Commutators of free random variables.
\emph{Duke Math. J.},
\textbf{92}, 
(1998), 
553--592.

\bibitem{NiSp06} A. Nica and R. Speicher. 
\emph{Lectures on the Combinatorics of Free Probability}.
London Mathematical Society Lecture Notes Series,
\textbf{335}, 
Cambridge University Press,
(2006).

\bibitem{Pa95} V. Pata.
L\'{e}vy type characterization of stable laws for free random variables.
\emph{Trans. Amer. Math. Soc.},
\textbf{347}, 
(1995),
2457--2472.

\bibitem{PAR07}
V. P{\'e}rez-Abreu and J. Rosi{\'n}ski.
Representation of infinitely divisible distributions on cones.
{\it J. Theoret. Probab.},
{\bf 20},
(2007),
535--544.

\bibitem{PeSa} V. P\'erez-Abreu and N. Sakuma.
Free infinite divisibility of free multiplicative mixtures of the Wigner distribution. 
\emph{J.  Theoret. Probab.}, 
\textbf{25},
(2012),
100--121.

\bibitem{RaSp07} N. Raj Rao and R. Speicher.
Multiplication of free random variables and the S-transform: The case of vanishing mean. 
\emph{Elect. Comm. Probab.},
\textbf{12},
(2007),
248--258.

\bibitem{S-Y} 
N.\ Saitoh and H.\ Yoshida.  
The infinite divisibility and orthogonal polynomials with a constant recursion formula in free probability theory.  
\emph{Probab.\ Math.\ Statist.}, \textbf{21}, Fasc.\ 1, (2001), 159--170.

\bibitem{Sa10} 
N. Sakuma. 
On free regular infinitely divisible distributions. 
{\it RIMS Kokyuroku Bessatsu},
{\bf  B27}, 
(2011),
115--122.

\bibitem{Sato} K. Sato. 
{\it L\'{e}vy Processes and Infinitely Divisible Distributions.}
Cambridge University Press, (1999). 

\bibitem{Sp94} R. Speicher.
Multiplicative functions on the lattice of noncrossing partitions and free convolution. 
\emph{Mathematische Annalen},
\textbf{298}, 
(1994), 
611--628.

\bibitem{S-W} R. Speicher and R. Woroudi.
Boolean convolution.
{\it Free Probability Theory, Ed. D. Voiculescu, Fields Inst.\ Commun.,}
{\bf 12}, 
Amer. Math. Soc.,
(1997), 
267--280. 

\bibitem{StVH} F. W. Steutel and K. Van Harn.
\emph{Infinite Divisibility of Probability Distributions on the Real Line. }Marcel-Dekker, New York, 2003.

\bibitem{Vo86} D.\ Voiculescu. 
Addition of certain non-commutative random variables. 
\emph{J.\ Funct.\ Anal.,} 
\textbf{66}, (1986), 323--346. 


\bibitem{Vo87b} D. Voiculescu.
Multiplication of certain non-commuting random variables.
\emph{J. Operator Theory},
\textbf{18}, 
(1987), 
223--235.

\bibitem{VoDyNi92} 
D. Voiculescu, K.\ Dykema and A. Nica. 
\emph{Free Random Variables.}
CRM Monograph Series, 
\textbf{1}, 
Amer. Math. Soc.,
(1992).
\end{thebibliography}
\end{document}